\documentclass[a4paper,10pt]{article}
\usepackage{amsfonts, amsthm, amsmath}

\theoremstyle{definition}
\newtheorem{example}{Example}
\theoremstyle{theorem}
\newtheorem{lemma}{Lemma}
\theoremstyle{definition}
\newtheorem{remark}{Remark}
\theoremstyle{definition}
\newtheorem{corollary}{Corollary}
\theoremstyle{Theorem}
\newtheorem{theorem}{Theorem}
\theoremstyle{definition}

\theoremstyle{Theorem}

\begin{document}

\title{Large Deviations for Heavy-Tailed Factor Models}
\author{Boualem Djehiche\thanks{boualem@math.kth.se}~ and Jens Svensson\thanks{jenssve@math.kth.se}\\ \\Department of Mathematics \\ Royal Institute of Technology\\ SE-100 44 Stockholm} 
\maketitle







\begin{abstract}
We study large deviation probabilities for a sum of dependent random variables from a heavy-tailed factor model, assuming that the components are regularly varying. We identify conditions where both the factor and the idiosyncratic terms contribute to the behaviour of the tail-probability of the sum. A simple conditional Monte Carlo algorithm is also provided together with a comparison between the simulations and the large deviation approximation. We also study large deviation probabilities for stochastic processes with factor structure. The processes involved are assumed to be L\'evy processes with regularly varying jump measures. 
Based on the results of the first part of the paper, we show that large deviations on a finite time interval are due to one large jump that can come from either the factor or the idiosyncratic part of the process.


\end{abstract}



\section{Introduction}
This paper is devoted to the study of large deviations of sums of dependent random variables and processes, where the dependence is generated through a factor model. Factor models are important in both financial theory and practice, because this form of structural dependence is both realistic and tractable.
From a theoretical point of view, different types of factor models give intuition to economic phenomena: the Capital Asset Pricing Model (CAPM) and the Arbitrage Pricing Theory (APT) are examples where factor structure is a fundamental property (see e.g. Cochrane (2001)). From an applied point of view, factor models are useful as approximations of other models and for dimension reduction. In many cases, reducing the number of dimensions of a model can make it tractable in practice. 

\medskip
Often, the random variables or vectors involved are assumed to be normally distributed, or at least \textit{light-tailed}. A random variable $X$ is called light-tailed if its tail-distribution $P(X>\lambda)$ tends to zero faster than $e^{-c\lambda}$ for some $c>0$.
 
Empirical studies of financial time series often conclude that data are \textit{heavy-tailed}, i.e.~the previous condition is not satisfied  (see e.g.~Cont (2001) for a review of some of these empirical findings). Consequently, light-tailed factor models may not be suited for describing the tail-properties of financial data. Therefore, it is of interest to incorporate the assumption of heavy tails into a factor model. 
As we will see, heavy-tailed factor models display qualitatively different behaviour from standard light-tailed models.

\medskip
In the first part of the paper, we restrict ourselves to the class of regularly varying random variables and vectors. This class is fairly rich and includes popular distributions such as Pareto and student's \textit{t}. See e.g.~Embrechts et al.~(1997) and Resnick (2004) for treatments of the univariate and multivariate case, respectively. 
\medskip
A random variable $X$ is regularly varying if there exist $\alpha\geq 0$ and $p\in [0,1]$ such that 
\begin{align}
&\lim_{x\to\infty}\frac{P({X}>tx)}{P({|X|}>x)}=pt^{-\alpha}\quad\text{and}\quad \lim_{x\to\infty}\frac{P({X}\leq-tx)}{P({|X|}>x)}=(1-p)t^{-\alpha},
\end{align}
for $t>0$.
We refer to $p$ as the tail balance parameter. The definition can also be formulated in terms of sequences instead of a continuous parameter $x$.
Clearly, regularly varying random variables are heavy-tailed according to the above definition.

\medskip
Since we will allow for dependence between factors, we also need the corresponding class of random vectors.
For random vectors, regular variation is defined through convergence of measures.
Specifically, an $\mathbb{R}^d$-valued random vector $\mathbf{X}$ is said to be regularly varying if 
there exist a sequence $a_n\to\infty$ and a measure $\mu$ on $\mathbb{R}^d$ such that
\begin{align}
\lim_{n\to\infty}n{P(a_n^{-1}\mathbf{X}\in B)}=\mu(B) \label{IntroMvRV}
\end{align}
and $\mu(B)<\infty$ for every Borel set $B\subset \mathbb{R}^d$  
satisfying $\mathbf{0}\notin \overline{B}$ and $\mu(\partial B)=0$, where $\overline{B}$ and $\partial B$ denote the closure and boundary of $B$, respectively. We write  $\mathbf{X}\in \textrm{RV}(\alpha,\mu)$. See Hult and Lindskog (2006) for details about equivalent definitions of regular variation.

\medskip
Using this class of distributions, we define a factor model for the vector $(R_1,\ldots, R_n)$ by letting
\begin{align}
R_i=\sum_{j=1}^d L_{ij}F_j+\varepsilon_i, \quad i=1,\ldots,n,\label{IntroModel}
\end{align}
where $\mathbf{F}_d=(F_1,\ldots, F_d)^{\mathrm{T}}$ is a regularly varying random vector, $\varepsilon_i$ are i.i.d.~regularly varying random variables and $\mathbf{L}_i=(L_{i1},\ldots,L_{id})$ are i.i.d.~random vectors.
All the random variables and vectors involved are assumed to be independent. The components of $\mathbf{F}_d$ are referred to as factors, $L_{ij}$ as factor loadings and $\varepsilon_i$ as idiosyncratic components.

A sum of variables from this model can be expressed as
\begin{align}
S_n=\sum_{i=1}^n R_i=\sum_{i=1}^{n}\sum_{j=1}^{d} L_{ij} F_j+\sum_{i=1}^n \varepsilon_i.
\label{IntroSum}
\end{align}
\noindent
The tail probability $P(S_n>\lambda)$ exhibits different asymptotic behaviour depending on the relation between the tail indices of the independent sums $\sum_{i=1}^n \varepsilon_i$ and $\sum_{i=1}^n \sum_{j=1}^d L_{ij} F_j$.

\medskip
Recall that (see e.g.~Embrechts et al.~(1997)) if two independent regularly varying random variables $X$ and $Y$ have different tail indices, $0<\alpha_X<\alpha_Y$, then
\begin{align*}
P(X+Y>\lambda)\sim P(X>\lambda),\quad \textrm{as } \lambda\to\infty, 
\end{align*}
which means that the random variable with heaviest tail, or smallest tail index, dominates the tail probability of the sum. 
On the other hand (see e.g.~Embrechts et al.~(1997)), if $X_1, X_2, \ldots$ are i.i.d.~regularly varying random variables with tail balance parameter $p$, we have with $n$ fixed, 
\begin{align}
P(\sum_{i=1}^nX_i>\lambda)\sim npP(|X_1|>\lambda), \quad \textrm{as } \lambda\to\infty,\label{IntroAsymptotic}
\end{align}
where $a(x)\sim b(x)$ as $x\to\infty$ denotes $\lim_{x\to\infty} a(x)/b(x)=1$. 
In fact, Relation \eqref{IntroAsymptotic} is still valid when $n\to\infty$ if $\lambda=\lambda_n$ increases sufficiently fast. Asymptotic probabilities of this kind are called large deviation probabilities. 

\medskip
For an appropriate choice of $\lambda_n$ we have
\begin{align}\label{IntroLD}
P(\sum_{i=1}^n X_i>\lambda_n)\sim npP(|X_1|>\lambda_n),\quad \textrm{as }n\to\infty.
\end{align}
We refer to Mikosch and Nagaev (1998) for details about the choice of sequence $\lambda_n$ under different distributional assumptions. 

\medskip
In this paper we consider regularly varying random variables with tail indices larger than 2, for which it was shown in Nagaev (1970) that if $\lambda_n$ is such that $\sqrt{n \log n}/\lambda_n\to 0$ as $n\to\infty$, then Relation \eqref{IntroLD} holds. Similarly, for tail probabilities of the sum $S_n$ given by \eqref{IntroSum}, we have two different situations. As $n\to\infty$ with $\lambda_n\sim n$, the tail behaviour of $S_n$ is determined by the tail probability of the sum $\sum_{i=1}^n \sum_{j=1}^d L_{ij} F_j$, whereas, when $\lambda\to\infty$ with $n$ fixed, it is determined by the sum with the heaviest tail.

\medskip\noindent
To obtain an expression where both sums contribute to the tail behaviour of $S_n$, we study the influence of the choice of $\lambda_n$ on the behaviour of large deviation probabilities of the form $P(S_n>\lambda_n)$, when $n\to\infty$. In the main result of the paper, Theorem \ref{LDTheorem}, we identify conditions under which there exists a sequence $\lambda_n$ such that both sums contribute to the large deviation probability of $S_n$. In particular, $\varepsilon_i$ should have heavier tail than $\mathbf{F}_d$. We also show that the i.i.d.~random vectors $\mathbf{L}_i$ only contribute through their expectations.

\medskip\noindent
Using the obtained results, we also study sums of heavy-tailed processes with factor structure. We adapt results from Hult and Lindskog (2005) to our case and derive a large deviation principle for our processes on $D([0,1], \mathbb{R}$), the space of real-valued c{\`a}dl{\`a}g functions on $[0,1]$. Here we note that extreme events during a finite time interval occur due to one large jump. Moreover, using \ref{LDTheorem}, we conclude that this large jump can come from either the factor or the idiosyncratic part of the process.

\medskip
The paper is organised as follows. In Section 2, we derive a large deviation result for sums of dependent random variables from a heavy-tailed factor model. Section 3 contains a numerical example where, under some further assumptions on the factor model, we derive a conditional Monte Carlo algorithm. Moreover, we compare the simulation results with the analytical approximations. 
Section 4 deals with large deviation results for heavy-tailed L\'evy processes with factor structure. Some proofs and technical results are collected in Section 5. 
 


\section{Large Deviations for Heavy-Tailed Factor Models}

In this section we investigate under which conditions both the factors and the idiosyncratic components in (\ref{IntroSum}) contribute to the large deviation probability $P(S_n> \lambda_n)$ as $n\to\infty$.

\medskip Consider the model given by (\ref{IntroModel}), which in matrix notation reads
\begin{align}
\mathbf{R}_n=\mathbf{\Lambda}_n\mathbf{F}_d+\boldsymbol{\varepsilon}_n\label{FactorModel},
\end{align}
where $\mathbf{\Lambda}_n$ denotes the matrix $(\mathbf{L}_i)_{i=1}^n$.
We assume that the vector of risk-factors $\mathbf{F}_d$ is regularly varying i.e.
\begin{align*}
\lim_{n\to\infty}{nP(a_n^{-1}\mathbf{F}_d\in B)}=\mu_{}(B),
\end{align*}
for Borel sets $B\subset \mathbb{R}^d$ satisfying $\mathbf{0}\notin \overline{B}$ and $\mu(\partial B)=0$, where $\mu$ is given and has tail index $\alpha_F>2$.
Furthermore, the rows of the matrix of factor loadings $\mathbf{\Lambda}_n$, $\mathbf{L}_i$, are independent copies of a random vector $\mathbf{L}=(L_1,\ldots,L_d)$ with $\mathbb{E}|L_j|^{\alpha_F+\delta}<\infty$ for $j=1,\ldots,d$ and some $\delta>0$.
\noindent The elements $\varepsilon_1,\ldots, \varepsilon_n $ of the idiosyncratic term are  i.i.d.~and regularly varying random variables with tail index $\alpha_{\varepsilon}>2$. 

\medskip
Denoting $S_{n,j}^{L}=\sum_{i=1}^n L_{ij}$, we get
\begin{align}
S_n=\sum_{j=1}^d S_{n,j}^{L}F_j+\sum_{i=1}^n \varepsilon_i \label{sumdecomp}.
\end{align}

\noindent By the law of large numbers, 
\begin{align*}
\lim_{n\to\infty}\frac{ S_{n,j}^L}{n}=\lim_{n\to\infty}\frac{1}{n}\sum_{i=1}^n L_{ij}= \mathbb{E}L_j\,\,\,\textrm{a.s.},
\end{align*} 
as $n\to\infty$, which suggests that
\begin{align}
P(\sum_{j=1}^d S_{n,j}^{L}F_j>\lambda_n x)\sim P(\sum_{j=1}^d (\mathbb{E} L_j) F_j>\frac{\lambda_n}{n} x),\quad \text{as}\,\, n\to\infty.
\end{align}

\noindent To verify this, we use Lemmas \ref{MVLimit} and \ref{MGIneq}, below. 
\begin{lemma}[]\label{MVLimit}
Let $\mathbf{X}$ be a $d\times 1$ regularly varying random vector, $\mathbf{X}\in \textrm{RV}(\alpha, \mu)$ and
let $\mathbf{A}_n\neq\mathbf{0}$ be a sequence of $1\times d$ random vectors independent of $\mathbf{X}$ such that $\mathbf{A}_n\to\mathbf{A}\neq\mathbf{0}$ a.s.,~as $n\to\infty$ and 
$\mathbb{E}(\sup_n |\mathbf{A}_n|_{\infty})^{\alpha+\delta}<\infty$, where, $|\mathbf{A}|_{\infty}=\sup_{|\mathbf{x}|=1}|\mathbf{A}\mathbf{x}|$.

\medskip\noindent
Then, for $0<\lambda_n\uparrow \infty$ and $x>0$, we have
\begin{align*}
\lim_{n\to \infty}\frac{P(\mathbf{A}_n \mathbf{X}  >\lambda_n x)}{P(|\mathbf{X}|>\lambda_n)} =  x^{-\alpha} \mu\left( \mathbf{A}^{-1} (1,\infty)\right).
\end{align*}
\end{lemma}
\begin{proof}See Section \ref{Proofs}.
\end{proof}

\begin{lemma}\label{MGIneq}
Let $X_i$, $i=1,2,\ldots$ be a sequence of i.i.d.~random variables ${E}|X_1|^{r}<\infty$, $r>1$. Then
${E}(\sup_k |\sum_{i=1}^k X_i|/k)^{r}<\infty$.
\end{lemma} 

\begin{proof} The result follows directly from the $L^p$ maximum inequality for martingales, see eg.~Durrett (1996).
\end{proof}

\medskip\noindent Lemma \ref{MGIneq} is needed to verify the conditions of Lemma \ref{MVLimit}.
Indeed, under the integrability assumptions $\mathbb{E}|L_j|^{\alpha_F+\delta}<\infty$ on $\mathbf{L}$, it follows that $\mathbb{E}|S^L_{n,j}/n|^{\alpha_F+\delta}<\infty$ and that $\mathbb{E}(\sup_k|S^L_{k,j}/k|)^{\alpha_F+\delta}<\infty$. Now, applying Lemma \ref{MVLimit}, we conclude that for fixed $x>0$ 
\begin{align*}
\lim_{n\to\infty}\frac{P(\sum_{j=1}^d S^L_{n,j}F_j>\lambda_nx)}{P(|\mathbf{F}_d|>\lambda_n/n)}=x^{-\alpha_F}\mu\left((\mathbb{E}\mathbf{L})^{-1}(1,\infty)\right).
\end{align*}

\medskip We now consider the tail-behaviour of the sum $S_n$. If $\mathbf{F}_d$ and $\varepsilon_1$ have the same tail indices, we expect $\mathbf{F}_d$ to dominate the extremal behaviour, i.e.~we expect the idiosyncratic components to become less relevant as $n$ grows due to the law of large numbers. Thus, the variation of the sum is mainly due to variation of the factors. If we want to use large deviation probabilities as approximations for finite $n$, we should try to avoid this behaviour. In the following Theorem, which is the main result of the paper, we state the behaviour of the tail probability of our sum under different assumptions. 

\begin{theorem}[]\label{LDTheorem}
Let $\mathbf{F}_d=(F_1,\ldots,F_d)$ be a regularly varying random vector, $\mathbf{F}_d\in\textrm{RV}(\alpha_F,\mu)$, and $\varepsilon_i$ be a sequence of i.i.d.~regularly varying random variables, $\varepsilon_i\in\textrm{RV}(\alpha_{\varepsilon})$, with tail balance parameter $p$. 
Consider the factor model given in \eqref{FactorModel} and the sum $S_n$ in Equation \eqref{sumdecomp}.\\
Let $\gamma_n\gg \rho_n$ denote $\lim_{n\to\infty} \gamma_n/\rho_n=\infty$.
\begin{itemize}
\item[(1)] If $\alpha_F\leq\alpha_{\varepsilon}$, then for any $\lambda_n\gg n$,
\begin{align*}
\lim_{n\to\infty}\frac{P(S_n>\lambda_nx)}{P(|\boldsymbol{F}_d|>\lambda_n/n)}=x^{-\alpha_F}\mu\left((E\mathbf{L})^{-1}(1,\infty)\right).\\
\end{align*}

\medskip
\item[(2)] Assume that $P(|\mathbf{F}_d|>x)=L_{|F|}(x)x^{-\alpha_F}$ and $P(|\varepsilon|>x)=L_{|\varepsilon|}(x) x^{-\alpha_{\varepsilon}}$, where $\alpha_F>\alpha_{\varepsilon}>2$. Define $\theta_F=(\alpha_F-1)\slash (\alpha_F-\alpha_{\varepsilon})$, $\theta_{\varepsilon}=\theta_F-1$. 
If $\alpha_F>\alpha_{\varepsilon}$, we have three different possibilities:\\
\\
\begin{itemize}
\item[(a)] If $\lambda_n\gg n^{\theta_F}$, then
\begin{align*}
\lim_{n\to\infty}\frac{P(S_n>\lambda_nx)}{nP(|{\varepsilon}|>\lambda_n)}=px^{-\alpha_{\varepsilon}}.
\end{align*}
\item[(b)] If $\lambda_n\ll n^{\theta_F}$, then
\begin{align*}
\lim_{n\to\infty}\frac{P(S_n>\lambda_nx)}{P(|\boldsymbol{F}_d|>\lambda_n/n)}=x^{-\alpha_F}\mu\left((E\mathbf{L})^{-1}(1,\infty)\right).
\end{align*}
\item[(c)] If $\lambda_n\sim n^{\theta_F}$, and
\begin{align}
\lim_{n\to\infty} \frac{L_{|\varepsilon|}(n^{\theta_F})}{L_{|F|}(n^{\theta_{\varepsilon}})}=C\in[0,\infty], \label{ConditionOnL}
\end{align}
then for $0\leq C<\infty$,
\begin{align}
\lim_{n\to\infty}\frac{P(S_n>\lambda_nx)}{P(|\mathbf{F}_d|>\lambda_n/n)}=x^{-\alpha_F}\mu\left((E\mathbf{L})^{-1}(1,\infty)\right)+x^{-\alpha_{\varepsilon}}pC\label{FactorDecomp2}
\end{align}
and for $C=\infty$,
\begin{align*}
\lim_{n\to\infty}\frac{P(S_n>\lambda_nx)}{nP(|{\varepsilon}|>\lambda_n)}=px^{-\alpha_{\varepsilon}}.
\end{align*}
\end{itemize}
\end{itemize}
\end{theorem}

\medskip
\begin{remark}Theorem \ref{LDTheorem} (c) provides a choice for $\lambda_n$ that, given the tail indices of $\mathbf{F}$ and $\varepsilon$,  yields the asymptotic behaviour \eqref{FactorDecomp2}.
Qualitatively, it also shows that for both parts to contribute to the large deviation behaviour, the idiosyncratic part must have heavier tail than the factors. 
\end{remark}

\begin{remark}
Condition \eqref{ConditionOnL} can be difficult to verify. The slowly varying functions of the norms are often not known, and are not easy to calculate explicitly. 
\noindent Examples where Condition \eqref{ConditionOnL} is satisfied include:
\begin{align*}
(a).& \,\, L_{|F|}(x)=c_1, \,\, L_{|\varepsilon|}(x)=c_2\\
(b).& \,\,L_{|F|}(x)\to c_1, \,\, L_{|\varepsilon|}(x)\to c_2\\
(c).& \,\,L_{|F|}(x)=a_1\log x +b_1, \,\, L_{|\varepsilon|}(x)=a_2\log x +b_2.
\end{align*}
\end{remark}

\begin{example}\label{ExampleFactor}
As an illustration of the application of Theorem \ref{LDTheorem}, we consider the case of independent Pareto-distributed factors and idiosyncratic components. Assume that $d=10$, i.e.~$\mathbf{F}_{10}=(F_1,\ldots,F_{10})$. We have $L_{|F|}=L_{|\varepsilon|}=1$ so that $C=1$.
Let $\alpha_F=5$ and $\alpha_{\varepsilon}=3$.
With $\lambda_n=n^{(5-1)/(5-3)}=n^2$ we obtain
\begin{align}
P(\sum_{i=1}^nR_i>\lambda_nx)&=P(\sum_{j=1}^{10} S_{n,j}^LF_j+\sum_{i=1}^n\varepsilon_i>\lambda_n x)\nonumber\\
&\sim P(\sum_{j=1}^{10} S_{n,j}^LF_j>\lambda_n x)+P( \sum_{i=1}^n \varepsilon_i>\lambda_n x)\nonumber\\
&\sim\sum_{j=1}^{10} P(\frac{S_{n,j}^L F_j}{n}>n x)+ n p P(\varepsilon_1>n^2 x)\nonumber\\
&\sim n^{-5}\left(\sum_{j=1}^{10} (EL_j)^{-5} x^{-5}+p x^{-3}\right).\label{ExampleFactorEqn}
\end{align}
\end{example}

\noindent Before proving Theorem \ref{LDTheorem}, we state a partial result.
\begin{lemma}[]\label{LDLemma}
Assume that $\mathbf{X}$ is a regularly varying $d$-dimensional random vector, $\mathbf{X}\in RV(\mu,\alpha_X)$, and $Y_i$ is a sequence of i.i.d.~regularly varying random variables, $Y_1\in\textrm{RV}(\alpha_Y)$, with tail balance parameter $p$.
Let $\mathbf{A}_n$  be a sequence of $d$-dimensional random vectors satisfying 
$E(\sup_n|\mathbf{A}_n|_{\infty})^{\alpha_X+\delta}<\infty$, for some $\delta>0$, and $\mathbf{A}_n \xrightarrow[]{\textrm{a.s.}} \mathbf{A}\neq \mathbf{0}$. Furthermore assume that $\mathbf{A}_n$, $Y_i$ and $\mathbf{X}$ are independent for all $i$ and $n$.

\medskip \noindent 
Consider the tail probabilities
\begin{align*}
&\overline{F}_{|X|}(x)=P(|\mathbf{X}|>x),\\
&\overline{F}^*(x)= P(n\mathbf{A}_n\mathbf{X}+\sum_{i=1}^n Y_i>x),\\
&\overline{F}_1(x)= P(n\mathbf{A}_n\mathbf{X}>x),\\
&\overline{F}_2(x)= P(\sum_{i=1}^n Y_i>x),
\end{align*}
where $x>0$.
Assume that there exists a sequence $\lambda_n\gg n$ such that 
\begin{align}
\lim_{n\to\infty} \frac{\overline{F}_2(\lambda_n x) }{\overline{F}_{|X|}(\lambda_n/n) }=Qx^{-\alpha_{Y}},\label{Assumption1}
\end{align}
where $Q\in[0,\infty]$.
Then,
\begin{align}
\lim_{n\to\infty} \frac{\overline{F}_1(\lambda_n )}{\overline{F}^*(\lambda_nx)}&=\frac{1}{x^{-\alpha_{X}}+x^{-\alpha_{Y}}Q/\mu_{A^{-1}}}\label{Limit1}
\end{align}
and
\begin{align}
\lim_{n\to\infty} \frac{\overline{F}_2(\lambda_n )}{\overline{F}^*(\lambda_nx)}&=\frac{1}{x^{-\alpha_{Y}}+x^{-\alpha_{X}} \mu_{A^{-1}}/Q}\label{Limit2},
\end{align}
where, $\mu_{A^{-1}}=\mu(\mathbf{A}^{-1}(1,\infty))$. If $Q$ is zero or infinite, we interpret the right hand side of relations \eqref{Limit1}-\eqref{Limit2} as limits.
\end{lemma}
\begin{proof}
See Section \ref{Proofs}.
\end{proof}

\begin{proof}[Proof of Theorem \ref{LDTheorem}]
We only derive Relation \eqref{FactorDecomp2}, the other relations are  proved in a similar fashion. First, we compute $Q$ in (\ref{Assumption1}). This gives us the sequence $\lambda_n$ via the tail indices. 
We then apply Lemma \ref{LDLemma} to obtain the results.

We have, with $\overline{F}_2(\lambda_n x)=P(\sum_{i=1}^n \varepsilon_i>\lambda_n x)$,
\begin{align*}
&\lim_{n\to\infty} \frac{\overline{F}_2(\lambda_n x) }{\overline{F}_{|F|}(\lambda_n/n) }\\
&\quad=\lim_{n\to\infty}\frac{\overline{F}_2(\lambda_n x) }{n\overline{F}_{|\varepsilon|}(\lambda_n)}\frac{n\overline{F}_{|\varepsilon|}(\lambda_n)}{\overline{F}_{|F|}(\lambda_n/n) }\\
&\quad=\lim_{n\to\infty}\underbrace{\frac{\overline{F}_2(\lambda_n x) }{n\overline{F}_{|\varepsilon|}(\lambda_n)}}_{I_1}\underbrace{\frac{L_{|\varepsilon|}(n^{\theta_F})}{L_{|F|}(n^{\theta_{\varepsilon}})}}_{I_2}\underbrace{\frac{n\lambda_n^{-\alpha_{\varepsilon}}}{(\lambda_n/n)^{-\alpha_{F}}}}_{I_3}.
\end{align*}

\noindent From \eqref{IntroLD} we get $I_1\to px^{-\alpha_{\varepsilon}}$ and, by assumption, $I_2\to C$. For simplicity, we restrict ourselves to the case $I_3=1$. This condition gives the expression for $\lambda_n$.
We then have $Q=pC$.
Applying Lemma \ref{LDLemma} we obtain, with $\mu_{L^{-1}}=\mu\big(\mathbf{L}^{-1}(1,\infty)\big)$,
\begin{align*}
&\lim_{n\to\infty}\frac{\overline{F}^*(\lambda_nx)}{\overline{F}_{|F|}(\lambda_n/n)}\\
&\quad=\lim_{n\to\infty}\frac{\overline{F}^*(\lambda_nx)}{\overline{F}_1(\lambda_n)}\frac{\overline{F}_1(\lambda_n)}{\overline{F}_{|F|}(\lambda_n/n)}\\
&\quad=(x^{-\alpha_{F}}+Qx^{-\alpha_{\varepsilon}}/\mu_{L^{-1}})\mu_{L^{-1}}=\mu_{L^{-1}} x^{-\alpha_{F}}+px^{-\alpha_{\varepsilon}}C,
\end{align*}
and we arrive at relation \eqref{FactorDecomp2}.

\end{proof}

The above results rely on the regular variation of the components involved. In the case of light-tailed random variables, the decomposition in Theorem \ref{LDTheorem} is no longer valid. We illustrate this in the following corollary by assuming light-tailed factors.

\begin{corollary}[]\label{LDLemmaLight}

Let $X>0$ be a light-tailed random variable with tail distribution $\overline{F}_X(x)\sim e^{-g(x)}$, where $g(x)-cx\to\infty$, as $x\to\infty$ for some $c>0$. Let $Y_i, i=1,2,\ldots$ be a sequence of i.i.d.~regularly varying random variables with tail-index $\alpha>0$, $Y_i\in RV(\alpha)$.
Then, for any sequence $\lambda_n$ such that $\lambda_n/n\to\infty$,
\begin{align*}
\lim_{n\to\infty} \frac{P(nX+\sum_{i=1}^n Y_i>\lambda_n)}{P(\sum_{i=1}^n Y_i>\lambda_n)}=1.
\end{align*}
\end{corollary}
\begin{proof}
Considering Equation \eqref{Assumption1}, we have
\begin{align*}
\frac{\overline{F}_2(\lambda_n ) }{\overline{F}_{|X|}(\lambda_n/n) }=\frac{P(\sum_{i=1}^n Y_i>\lambda_n)}{P(X>\lambda_n/n)}\sim\frac{e^{-g(\lambda_n/n)}}{n\lambda_n^{-\alpha}},
\end{align*}
so that
\begin{align*}
\log{Q}=\lim_{n\to\infty} g(\lambda_n/n)+\log{n} -\alpha \log{ \lambda_n}=\infty.
\end{align*}
Hence, using Equation \eqref{Limit2} we obtain the result. 
\end{proof}



\section{Simulation} 
To see how the approximations derived in the previous section behave, we will present a short simulation study.
Since tail probabilities are rare events, naive Monte Carlo Simulation can be very slow.
To achieve a given relative error, a huge number of simulations are often needed. Methods of variance reduction are therefore crucial for obtaining a satisfactory estimation. 
We present a method for estimating the tail probability of a sum of variables from our factor model, under certain restrictive conditions.

\medskip Variance reduction algorithms for sums of heavy-tailed random variables are often based on the observation that, asymptotically, a sum is determined by its largest term. This is then used for conditioning or change of measure, \textit{importance sampling}. Examples of such algorithms include Juneja et al.~(2002), where measures for importance sampling are chosen by the so-called hazard rate twisting method. Dupuis et al.~(2006) use a dynamic algorithm to change measure for each term in the sum, making sure that the rare event in question occurs. 
In the setting of a portfolio loss depending on multivariate $t$-distributed risk factors, Glasserman et al.~(2002) derive an importance sampling algorithm using a quadratic approximation of the portfolio loss.

\medskip 
Using the conditioning approach suggested in Asmussen and Kroese (2006) we can state a simulation algorithm for our factor model with i.i.d.~factors, i.i.d.~loadings and i.i.d.~idiosyncratic components. 

\medskip\noindent Denoting $M_{n,d}=\max(\varepsilon_1,\ldots,\varepsilon_n,S_{n,1}^L F_1,\ldots,S_{n,d}^L F_d)$ and assuming that the all variables are continuous, we have
\begin{align*}
P(S_n>x)&=P(\sum_{j=1}^d S_{n,j}^L F_j+\sum_{i=1}^n \varepsilon_i>x)=P(S^F_d+S^{\varepsilon}_n>x)\\
&=nP(S_n>x, M_{n,d}=\varepsilon_n)+dP(S_n>x, M_{n,d}=S_{n,d}^L F_d).
\end{align*}

Conditioning yields
\begin{align*}
&P(S_n>x, M_{n,d}=\varepsilon_n)\\
&\quad=EP(S_n>x, M_{n,d}=\varepsilon_n|\varepsilon_1,\ldots,\varepsilon_{n-1}, S_{n,1}^L F_1,\ldots,S_{n,d}^L F_d)\\
&\quad=EP(\varepsilon>(x-S_{n-1})\vee M_{n-1,d}|\varepsilon_1,\ldots,\varepsilon_{n-1}, S_{n,1}^L F_1,\ldots,S_{n,d}^L F_d).\\
&\textrm{Similarly,}\\
&P(S_n>x, M_{n,d}=S_{n,d}^L D_d)\\
&\quad= EP(S_n>x, M_{n,d}=S_{n,d}^L F_d|\varepsilon_1,\ldots,\varepsilon_{n}, S_{n,1}^L F_1,\ldots,S_{n,d-1}^L F_{d-1})\\
&\quad= EP(S_{n,d}^L F_d>(x-S_{n-1})\vee M_{n,d-1}|\varepsilon_1,\ldots,\varepsilon_{n}, S_{n,1}^L F_1,\ldots,S_{n,d-1}^L F_{d-1}).
\end{align*}
If the distributions of $\varepsilon$ and $S_{n,d}^L F_d$ are known, these probabilities can be calculated explicitly.
Alternatively, conditioning on $\boldsymbol{\Lambda}_n$ and calculating the last probability by simulation only requires knowledge of the marginal distribution of $F_d$.

\begin{table}[h!t]
\centering 
\begin{tabular}{c c c c c c c c} 
\hline\\[-2ex] 
 $x$ & & $n$ &  $10^3$ & $10^4$  & $10^5$   & & \\ [0.5ex] 
\hline\\[-1ex] 
0.1& &   & 1.0010e-09  & 1.0010e-14&   1.0010e-19& & LD-Estimate\\ 
& &   &1.9878e-09  &  1.0673e-14& 1.0074e-19 & &Conditional MC\ \\[1ex]
1& &   &   1.1000e-14&  1.1000e-19 & 1.1000e-24 & & \\
& &   & 1.1708e-14 & 1.1068e-19  &  1.1007e-24 & & \\[1ex]
10& &  & 1.1000e-18& 1.1000e-23   &  1.1000e-28
 &  & \\ 
& &  &  1.1049e-18 & 1.1005e-23   & 1.1000e-28 &  & \\ [1ex]
\hline\\[-1ex] 

\end{tabular} \caption{Estimates of ${P}(S_n>\lambda_n x)$ using conditional Monte Carlo for the model in Example \ref{ExampleFactor} with $\lambda_n=n^2$. The number of factors is $d=10$ and $L_{ij}=1$.
  10000 iterations are used for all estimates. The LD-estimate uses Equation \eqref{ExampleFactorEqn} from Example \ref{ExampleFactor}.}  \label{Table1} 

\end{table}
\medskip
In Table \ref{Table1}, we compare the analytical approximation of the tail probability in Example 1 to simulations using the above algorithm. Since it is a large deviation result, the approximation performs best when we consider regions far out in the tail, i.e.~when $\lambda_n x=n^2 x$ is large. The resulting probabilities in these regions range from small to extremely small.  As expected,  we obtain the worst results for $x=0.1$ and $n=10^3$.






\section{Large Deviations for Factor Processes}
In this section, we study the large deviation behaviour of sums of heavy-tailed processes with factor structure.
We assume that, in Equation \eqref{IntroSum}, $\mathbf{F}_d=\{\mathbf{F}_d(t): t\in[0,1]\}$ and $\varepsilon_i=\{\varepsilon_i(t): t\in[0,1]\}$ are L\'evy processes, whose increments are regularly varying, or equivalently, whose L\'evy measures are regularly varying. 

\medskip\noindent In Theorem \ref{Prop} we establish a large deviation result for the process
\begin{align*}
S_n(t)=\sum_{j=1}^d S_{n,j}^L F_j(t)+\sum_{i=1}^n \varepsilon_i(t).
\end{align*}

\begin{theorem}\label{Prop}
Assume that $\mathbf{F}_d(t)$ is a $d$-dimensional L\'evy process and that $\varepsilon_i(t), \,i=1,\ldots,n$ are i.i.d.~L\'evy processes. Furthermore, assume that their L\'evy measures are regularly varying with tail indices satisfying $\alpha_F>\alpha_{\varepsilon}>2$. 

\noindent Let
\begin{align*}
P(|\mathbf{F}_d(1)|>x)&=L_{|F|}(x)x^{-\alpha_F},\\
P(|\varepsilon(1)|>x)&=L_{|\varepsilon|}(x)x^{-\alpha_{\varepsilon}},
\end{align*}
and assume that $L_{|F|}(x)$ and $L_{|\varepsilon|}(x)$ satisfy condition \eqref{ConditionOnL} in Theorem 1.  

\noindent Then,
\begin{align}\label{LDPprocess}
\lim_{n\to\infty} \gamma_n P(\lambda_n^{-1}S_n\in B)=\tilde{m}(B),
\end{align} 
for all Borel sets $B\in D([0,1],\mathbb{R})$  with $0\notin \overline{B}$ and $\tilde{m}(\partial B)=0$. We denote this property by $S_n\in \textrm{LD}((\gamma_n, \lambda_n),\tilde{m}, D([0,1],\mathbb{R}))$.
 
\medskip\noindent Moreover, $\tilde{m}$ puts all mass on step functions with one step, i.e.
\begin{align*}
\tilde{m}(\mathcal{V}_0^c)&=0,
\end{align*}
where $\mathcal{V}_0=\{{x}\in D([0,1],\mathbb{R}): {x}={y}1_{[v,1]}, v\in[0,1], {y}\in\mathbb{R}\}$. That is, any extreme event during the interval is due to one large jump of either the factor or the idiosyncratic part of the process.
\end{theorem}

\medskip\noindent
The proof of Theorem \ref{Prop} is given in Section 5, below. We end this section with an example.
\begin{example}
Let $\mathbf{F}_d(t)$ and $\varepsilon(t)$ be compound Poisson processes
\begin{align*}
\mathbf{F}_d(t)&=\sum_{i=1}^{N^F_t}\mathbf{Z}_i\\
\varepsilon_i(t)&=\sum_{i=j}^{N^{\varepsilon}_t} {W}_{ij},
\end{align*}
where $\mathbf{Z}_i=(Z_i^1,\ldots,Z_i^d)$ are random vectors with i.i.d.~components such that $P(|Z_1^1|>x)=x^{-\alpha_F}$ with tail balance parameter $p_F$ and $W_{ij}$ are i.i.d.~random variables such that $P(|W_{11}|>x)=x^{-\alpha_{\varepsilon}}$ with tail balance parameter $p_{\varepsilon}$. $N^F_t$ and $N^{\varepsilon}_t$ are Poisson processes with intensities $\lambda_F$ and $\lambda_{\varepsilon}$, respectively. Assume that the tail-indices satisfy $\alpha_F>\alpha_{\varepsilon}>2$. Both  $\mathbf{F}_d(1)$ and $\varepsilon_i(1)$ are regularly varying, and with $|\cdot|=|\cdot|_1$, we have
\begin{align*}
P(|\mathbf{F}_d(1)|>x)\sim d \lambda_F P(|Z_{11}|>x) \textrm{ and }
 P(|\varepsilon_1(1)|>x)\sim  \lambda_{\varepsilon} P(|W_{11}|>x).
\end{align*}
The conditions of Theorem \ref{Prop} being satisfied, we get $\gamma_nP(S_n\in \lambda_n B)\to \tilde{m}(B)$, where $\tilde{m}$ puts all its mass on step functions with one step. Moreover,
\begin{align*}
m_t(x,\infty):=\lim_{n\to\infty} \gamma_n P(\lambda_n^{-1}S_n(t)\in (x,\infty))
\end{align*} is explicitly given by (see (\ref{FactorDecomp2}), above)
\begin{align*}
m_t(x,\infty)=tp_F\sum_{j=1}^d (EL_j)^{-\alpha_F}x^{-\alpha_F} +tp_{\varepsilon}\frac{\lambda_{\varepsilon}}{d\lambda_F}x^{-\alpha_{\varepsilon}}.
\end{align*}
\end{example}



\section{Proofs and Technical Results}\label{Proofs}

To prove Lemma \ref{MVLimit}, we use the following multivariate version of Breiman's Lemma proved by Basrak, Davis and Mikosch (2002).
\begin{lemma}[Breiman's lemma]
Let $\mathbf{X}$ be a $d\times 1$ regularly varying random vector and let $\mathbf{A}$ be a $k\times d$ random matrix, independent of $\mathbf{X}$. If $0<E|\mathbf{A}|_{\infty}^{\alpha+\delta}<\infty$ for some $\delta>0$, then
\begin{align*}
\lim_{n\to\infty}\frac{P(\mathbf{A}\mathbf{X}\in a_n B)}{P(|\mathbf{X}|>a_n)}=E (\mu \circ \mathbf{A}^{-1} (B)).
\end{align*}
for any Borel set $B\subset \mathbb{R}^d$ satisfying $\mathbf{0}\notin \overline{B}$ and $\mu(\partial B)=0$.
\end{lemma}

\begin{proof}[Proof of Lemma \ref{MVLimit}]
We split $\mathbf{X}$ into positive and negative parts, $\mathbf{X}=\mathbf{X}^+-\mathbf{X}^-$, where $\mathbf{X}^+=(X_1^+,\ldots, X_d^+)$, $\mathbf{X}^-=(X_1^-,\ldots, X_d^-)$. 
The infimum and supremum of the vector $\mathbf{A}_k$ is interpreted component-wise, i.e. $\sup_{k>M}\mathbf{A}_k=(\sup_{k>M}A_k^1,\ldots,\sup_{k>M}A_k^d)$ and analogously for the infimum.
Fix $M>0$. For $n>M$ we have, 
\begin{align}
P(\mathbf{A}_n \mathbf{X} >\lambda_n x)&=P(\mathbf{A}_n (\mathbf{X}^+ -\mathbf{X}^-)>\lambda_nx)\nonumber\\
&\leq P(\sup _{k>M}\mathbf{A}_k\mathbf{X}^+ -\inf_{k>M}\mathbf{A}_k\mathbf{X}^->\lambda_n x)\nonumber\\
&=P((\sup _{k>M}\mathbf{A}_k,\inf_{k>M}\mathbf{A}_k) (\mathbf{X}^+,-\mathbf{X}^-)^{\textrm{T}}>\lambda_n x)\label{infsupUpper}.
\end{align}

\noindent The same argument also provides a lower bound, 
\begin{align}
P(\mathbf{A}_n \mathbf{X} >\lambda_n x)\geq P((\inf _{k>M}\mathbf{A}_k,\sup_{k>M}\mathbf{A}_k) (\mathbf{X}^+,-\mathbf{X}^-)^{\textrm{T}}>\lambda_n x)\label{infsupLower}.
\end{align}
The probability $P(\mathbf{A}_n \mathbf{X}  >\lambda_n x)/P(|\mathbf{X}|>\lambda_n)$ is thus bounded from above and below.
To determine these bounds, we need to show regular variation of the vector $(\mathbf{X}^+,-\mathbf{X}^-)^{\textrm{T}}$.

\medskip
Let $E_1=\overline{\mathbb{R}}^d\backslash\{\mathbf{0} \}$ and $E_2=\{\mathbf{z}'\in\overline{\mathbb{R}}^{2d}\backslash\{\mathbf{0} \}:\mathbf{z}'=(\mathbf{z}^+ ,-\mathbf{z}^- )^{\textrm{T}},\mathbf{z}\in\overline{\mathbb{R}}^d\backslash\{\mathbf{0} \}\} $ and define the continuous transformation 
\begin{align*}
T:E_1&\longrightarrow E_2\\
\mathbf{x}& \longmapsto (\mathbf{x}^+, -\mathbf{x}^-)^{\textrm{T}}.
\end{align*}
Any relatively compact set $K_2$ of $E_2$ is of the form
\begin{align*}
K_2=\{\mathbf{z}'=(\mathbf{z}^+,-\mathbf{z}^-)\in \overline{\mathbb{R}}^{2d}\backslash\{\mathbf{0} \}: \mathbf{z}\in \overline{\mathbb{R}}^{d}\backslash\{\mathbf{0} \} \}, 
\end{align*}
 bounded away from $\mathbf{0}$, i.e.~$\mathbf{0}\notin\overline{K}_2$. Since $\mathbf{z}'\neq \mathbf{0}\Rightarrow\mathbf{z}\neq \mathbf{0}$, it is obvious that the inverse images of these sets in $\overline{\mathbb{R}}^{d}\backslash\{\mathbf{0} \}$ are bounded away from $\mathbf{0}$ as well.

\medskip Hence, if $K_2$ is compact in $\overline{\mathbb{R}}^{2d}\backslash\{\mathbf{0} \}$  then  $K_1=T^{-1}(K_2)$ is compact in \mbox{$\overline{\mathbb{R}}^d\backslash\{\mathbf{0} \}$}. Therefore, vague convergence of a sequence of measures $\mu_n$ on $E_1$ implies vague convergence of the induced measures $\hat{\mu}_n=\mu_n\circ T^{-1}$ on $E_2$. Specifically, since $|T(\mathbf{x})|=|\mathbf{x}|$ and $T(a\mathbf{x})=aT(\mathbf{x})$ for any $a>0$,
\begin{align*}
\frac{P(T(\mathbf{X})  \in \lambda_n B)}{P(|T(\mathbf{X})|>\lambda_n)}=\frac{P(\mathbf{X}  \in T^{-1} (\lambda_n B))}{P(|\mathbf{X}|>\lambda_n)}=\frac{P(\mathbf{X}  \in\lambda_n T^{-1}(B))}{P(|\mathbf{X}|>\lambda_n)}\xrightarrow[]{v}\mu(T^{-1}(B)),
\end{align*}
Therefore, the vector $T(\mathbf{X})=(\mathbf{X}^+,-\mathbf{X}^-)^{\textrm{T}}$ is regularly varying. 

\medskip
Since, $E\sup_n |\mathbf{A}_n|_{\infty}<\infty$ it follows that $E|(\sup _{k>M}\mathbf{A}_k,\inf_{k>M}\mathbf{A}_k)|_{\infty}<\infty$, so we can use the multivariate version of Breiman's lemma to determine the bounds (\ref{infsupUpper}) and (\ref{infsupLower}). This yields
\begin{align}
  &E\big(\mu\circ (\inf_{k>M} \mathbf{A}_k,\sup_{k>M} \mathbf{A}_k )^{-1}(1,\infty)\big) x^{-\alpha}\nonumber \\
& \quad \leq \liminf_{n\to\infty}\frac{P(\mathbf{A}_n \mathbf{X} >\lambda_n x)}{P(|\mathbf{X}|>\lambda_n)}  \leq \limsup_{n\to\infty}\frac{P(\mathbf{A}_n \mathbf{X} >\lambda_n x)}{P(|\mathbf{X}|>\lambda_n)} \label{bounds3}\\
&\quad\leq  E\big(\mu\circ (\sup_{k>M} \mathbf{A}_k,\inf_{k>M} \mathbf{A}_k)^{-1}(1,\infty)\big) x^{-\alpha}\nonumber.
\end{align}
Since $\mathbf{A}_n\xrightarrow[n\to\infty]{\textrm{a.s.}}\mathbf{A}$ we have $\inf_{k>M}\mathbf{A}_k\xrightarrow[M\to\infty]{\textrm{a.s.}} \mathbf{A}$ and $\sup_{k>M}\mathbf{A}_k\xrightarrow[M\to\infty]{\textrm{a.s.}} \mathbf{A}$. 
It remains to verify that we can evaluate these limits inside the expectations. We have
\begin{align*}
\mu\circ (\inf_{k>M} \mathbf{A}_k, \sup_{k>M} \mathbf{A}_k)^{-1}(1,\infty)&\leq \mu\circ (\sup_{k>M} \mathbf{A}_k,\inf_{k>M} \mathbf{A}_k)^{-1}(1,\infty)\\
&\leq\mu\circ (\sup_{k} \mathbf{A}_k,\inf_{k} \mathbf{A}_k)^{-1}(1,\infty)
\end{align*}
and
\begin{align*}
&E\mu\circ (\sup_{k} \mathbf{A}_k,\inf_{k} \mathbf{A}_k)^{-1}(1,\infty)\\
&\quad=\,E\mu(\mathbf{z}\in\mathbb{R}^d:(\sup_{k} \mathbf{A}_k,\inf_{k} \mathbf{A}_k)(\mathbf{z}^+,-\mathbf{z}^- )^{\textrm{T}}>1)\\
&\quad\leq\,E\mu(\mathbf{z}\in\mathbb{R}^d:(\sup_{k} |\mathbf{A}_k|_{\infty})\mathbf{1}_{2d}^{\textrm{T}} (\mathbf{z}^+,\mathbf{z}^- )^{\textrm{T}}>1)\\
&\quad=\,E(\sup_{k} |\mathbf{A}_k|_{\infty})^{\alpha}\mu(\mathbf{z}\in\mathbb{R}^d:\mathbf{1}_{2d}^{\textrm{T}} (\mathbf{z}^+,\mathbf{z}^- )^{\textrm{T}}>1)\\
&\quad=\,E(\sup_{k} |\mathbf{A}_k|_{\infty})^{\alpha}\mu(\mathbf{z}\in\mathbb{R}^d:\mathbf{1}_{d}^{\textrm{T}}|\mathbf{z}|>1)<\infty,
\end{align*}
with $|\mathbf{z}|=(|z_1|,\ldots,|z_d|)$.
Hence, by Dominated Convergence,
\begin{align*}
\lim_{M\to\infty}&E\mu(\mathbf{z}\in\mathbb{R}^d:(\sup_{k>M} \mathbf{A}_k,\inf_{k>M} \mathbf{A}_k)(\mathbf{z}^+,-\mathbf{z}^- )^{\textrm{T}}>1)\\
&=E\mu(\mathbf{z}\in\mathbb{R}^d:(\mathbf{A}, \mathbf{A})(\mathbf{z}^+,-\mathbf{z}^- )^{\textrm{T}}>1)\\
&=E\mu(\mathbf{z}\in\mathbb{R}^d:\mathbf{A}\mathbf{z}>1).
\end{align*}
A similar calculation applies to the lower bound in equation (\ref{bounds3}), with the same limit. Letting $M\to\infty$ in that equation yields the conclusion.
\end{proof}





\begin{proof}[Proof of Lemma \ref{LDLemma}]
We first note that if $U$ and $V$ are independent random variables, we have
 \begin{align*}
  P(U+V>x)\ge P(U>(1+\delta)x)P(|V|<\delta x)+P(|U|<\delta x)P(V>(1+\delta)x).
\end{align*}
Therefore, setting $U=n\mathbf{A}_n\mathbf{X}\,\,$ and $\,\, V=\sum_{i=1}^nY_i$, we get
\begin{align}\label{bound1}
\overline{F}^*(x)&\geq \big(\overline{F}_1((1+\delta)x)P(|\sum_{i=1}^nY_i|<\delta x) \nonumber\\
&+\overline{F}_2((1+\delta)x)P(|n\mathbf{A}_n\mathbf{X}|<\delta x)\big).
\end{align}
Furthermore, since for $\delta\in(0,1/2)$ we have
\begin{align*}
 \{U+V>x\} \subset \{U>(1-\delta)x \}\cup \{V>(1-\delta)x\}\cup\{U>\delta x, V>\delta x \},
\end{align*}
it follows that
\begin{align}\label{bound2}
\overline{F}^*(x)\leq \overline{F}_1((1-\delta)x)+\overline{F}_2((1-\delta)x)+\overline{F}_1(\delta x)\overline{F}_2(x).
\end{align}
Relation (\ref{Limit1}) is then obtained by dividing both sides in \eqref{bound1} and \eqref{bound2} by $\overline{F}_1(\lambda_n )$, and inverting.

\medskip
The lower bound consists of two parts. The first part is
\begin{align*}
&\lim_{n\to\infty}\frac{\overline{F}_1((1+\delta)x\lambda_n )}{\overline{F}_1(\lambda_n )}P(|\sum_{i=1}^nY_i)|<\delta\lambda_nx)\\
&\quad=\lim_{n\to\infty}\frac{\overline{F}_1((1+\delta)\lambda_n x)}{\overline{F}_{|X|}(\lambda_n/n)}\frac{\overline{F}_{|X|}(\lambda_n/n)}{\overline{F}_1(\lambda_n )}P(|\sum_{i=1}^nY_i)|<\delta\lambda_nx)\\
&\quad=\,\,x^{-\alpha_X}(1+\delta)^{-\alpha_X},
\end{align*}
where we have used Lemma \ref{MVLimit} and the fact that $n /\lambda_n\to 0$, as $n\to\infty$, i.e. $\lambda_n$ is in the large deviation region which imlpies that (cf.~Proposition 3.1 in Mikosch and Nagaev (1998))  
\begin{align*}
\lim_{n\to\infty}P(|\sum_{i=1}^nY_i)|<\delta\lambda_n)=1
\end{align*}
and
\begin{align*}
\lim_{n\to\infty}P(|n\mathbf{A}_n\mathbf{X}|<\delta \lambda_n)=1.
\end{align*}

The second part is
\begin{align*}
&\lim_{n\to\infty} \frac{\overline{F}_2((1+\delta)\lambda_nx )}{\overline{F}_1(\lambda_n )}P(|n\mathbf{A}_n\mathbf{X}|<\delta\lambda_nx)\\
&\quad=\lim_{n\to\infty} \frac{\overline{F}_2((1+\delta)\lambda_n x)}{\overline{F}_{|X|}(\lambda_n/n)} \frac{\overline{F}_{|X|}(\lambda_n/n)}{\overline{F}_1(\lambda_n )}P(|n\mathbf{A}_n\mathbf{X}|<\delta\lambda_nx)\\
&\quad=\,\,Q\big((1+\delta)x\big)^{-\alpha_{Y}}\big(\mu(\mathbf{A}^{-1}(1,\infty))\big)^{-1},
\end{align*}
using Assumption (\ref{Assumption1}) and Lemma \ref{MVLimit}.

\medskip
The upper bound is treated similarly, although it consists of three parts.
The first part is treated using Lemma \ref{MVLimit} as above.
The second part is
\begin{align*}
&\lim_{n\to\infty} \frac{\overline{F}_2((1-\delta)\lambda_n x)}{\overline{F}_1(\lambda_n )}\\
&\quad=\lim_{n\to\infty} \frac{\overline{F}_2((1-\delta)\lambda_nx )}{\overline{F}_{|X|}(\lambda_n/n)} \frac{\overline{F}_{|X|}(\lambda_n/n)}{\overline{F}_1(\lambda_n )}\\
&\quad=\,\,Q\big((1-\delta)x\big)^{-\alpha_{Y}}\big(\mu(\mathbf{A}^{-1}(1,\infty))\big)^{-1}.
\end{align*}
The third and last part is
\begin{align*}
&\lim_{n\to\infty} \underbrace{\frac{\overline{F}_1(\delta\lambda_n z)}{\overline{F}_1(\lambda_n z)}}_{\to (z\delta)^{-\alpha_X}} \underbrace{\overline{F}_2(\delta\lambda_n )}_{\to 0}=0.
\end{align*}
Hence, with $\mu_{A^{-1}}=\mu(\mathbf{A}^{-1}(1,\infty))$, it follows that
\begin{align*}
\frac{1}{\big((1-\delta)z\big)^{-\alpha_X}+Q\big((1-\delta)x\big)^{-\alpha_{Y}}\slash\mu_{A^{-1}}}&\leq\liminf_{n\to\infty} \frac{\overline{F}_1(\lambda_n)}{\overline{F}^*(\lambda_nx)}\\ & \leq\limsup_{n\to\infty} \frac{\overline{F}_1(\lambda_n)}{\overline{F}^*(\lambda_nx)}  \\&\leq \frac{1}{\big((1+\delta)x\big)^{-\alpha_X}+Q\big((1+\delta)x\big)^{-\alpha_{Y}}\slash\mu_{A^{-1}}}.
\end{align*}
Letting $\delta\to0$ proves the first relation.
The second relation is shown analogously.
\end{proof}


The following proof of Theorem \ref{Prop} relies on several results from the work by Hult and Lindskog (2005), adapted to our conditions. All the arguments in their proofs apply, with obvious modifications.
\begin{proof}[Proof of Theorem \ref{Prop}]
By Theorem \ref{LDTheorem}, we have that
\begin{align*}
\lim_{n\to\infty} \gamma_n P(\lambda_n^{-1}S_n(1)>x)=\tilde{\mu}(x,\infty),
\end{align*}
where $\tilde{\mu}$ is given by \eqref{FactorDecomp2} and $\gamma_n^{-1}=P(|\mathbf{F}_d(1)|>\lambda_n/n)$.
Since both $\mathbf{F}_d$ and $\varepsilon$ are L\'evy-processes, we also have
\begin{align*}
\lim_{n\to\infty} \gamma_n P(\lambda_n^{-1}S_n(t)>x)=t\tilde{\mu}(x,\infty)
\end{align*}
for every $t\in[0,1]$. Furthermore, 
\begin{align*}
\tilde{m}_{\delta}(B_{0,\varepsilon}^c)-\tilde{m}_{0}(B_{0,\varepsilon}^c)&=\delta\tilde{\mu}(y\in\mathbb{R}: |y|>x)\\
\tilde{m}_{1}(B_{0,\varepsilon}^c)-\tilde{m}_{1-\delta}(B_{0,\varepsilon}^c)&=\delta\tilde{\mu}(y\in\mathbb{R}: |y|>x).
\end{align*}
Finally, we have
\begin{align*}
\alpha^n_{\lambda_n,1}(1)&=\sup\{P^n_{s,t}(x,B_{x,\lambda_n}^c):x\in\mathbb{R};s,t\in[0,1];t-s\in[0,1] \}\\
&=P(|S_n(1)-0|>\lambda_n)\to0,
\end{align*} 
as $n\to\infty$, since $\lambda_n$ is in the large deviation region.
The conditions of Theorem 13 in Hult and Lindskog (2005) are hence satisfied. This proves the first part of Proposition \ref{Prop}.
It remains to show that $\tilde{m}$ puts all its mass on step functions with one step. 

\medskip Let $B(p,\epsilon,[0,1])=\{\mathbf{x}\in D([0,1],\mathbb{R}^d): \mathbf{x}\textrm{ has $\epsilon$-oscillation $p$ times in }[0,1]\}$, where, for $\epsilon>0$ and $p$ a positive integer, the process $\mathbf{x}\in D([0,1],\mathbb{R}^d)$ is said to have $\epsilon$-oscillation $p$ times in $[0,1]$ if there exist $t_0,\ldots,t_p\in [0,1]$ with $t_0<\ldots <t_p$ such that $|\mathbf{x}_{t_i} -\mathbf{x}_{t_{i-1}}|>\epsilon$ for $1=1,\ldots,p$.

\medskip\noindent Using Lemma 21 in Hult and Lindskog (2005), we get
\begin{align*}
\liminf_{n\to\infty} \gamma_n P(S_n\in B(2,\lambda_n \epsilon,[0,1]))=0.
\end{align*}
Since the convergence of $\gamma_n P(\lambda_n^{-1} S_n\in B)$ to $\tilde{m}(B)$ is equivalent to  
$$
\liminf_{n\to\infty} \gamma_n P(\lambda_n^{-1} S_n\in G)\geq \tilde{m}(G)
$$ 
for all open and bounded $G$, and $G=B(2,\epsilon,[0,1])$ is open, we have that $\tilde{m}(B(2,\epsilon,[0,1]))=0$ for all $\epsilon>0$. It follows that
\begin{align*}
\tilde{m}(\bigcup_{\epsilon\in\mathbb{Q}, \epsilon>0} B(2,\epsilon,[0,1]))=0.
\end{align*}
Using that
\begin{align*}
\mathcal{V}_0^c \subset \bigcup_{\epsilon\in\mathbb{Q}, \epsilon>0} B(2,\epsilon,[0,1])),
\end{align*}
we conclude that $\tilde{m}(\mathcal{V}_0^c)\leq \tilde{m}(\bigcup_{\epsilon\in\mathbb{Q}, \epsilon>0} B(2,\epsilon,[0,1]))=0$.
\end{proof}

\medskip
\medskip

\section*{Acknowledgements}
The authors would like to thank Allan Gut and Filip Lindskog for valuable discussions as well as comments on the paper. Financial support from the G\"oran Collert Foundation is gratefully acknowledged.

\end{document}